\documentclass[12pt]{article}
\usepackage{times}
\usepackage{hyperref}

\usepackage{latexsym}
\usepackage{amsmath}
\usepackage[left=1in,top=1in,right=1in]{geometry}
\usepackage{amssymb,amsmath}
\usepackage{amsthm}
\usepackage{amsmath}
\usepackage{amsfonts}
\usepackage{titlesec}
\titlelabel{\thetitle.\quad}
\newtheorem{definition}{Definition}[section]
\newtheorem{corollary}{Corollary}[section]
\newtheorem{theorem}{Theorem}[section]
\newtheorem*{keywords*}{Keywords}

\newtheorem{proposition}{Proposition}[section]

\newcommand{\GQE}{generalized quasi-Einstein} 
\newcommand{\NCPM}{normal metric contact pair manifold}
\title{\textbf{ {\Large  Generalized Quasi-Einstein Manifolds in Contact Geometry} }}
\author{\.Inan \"Unal\\ Department of Computer Engineering\\Munzur University, Tunceli, Turkey\\ inanunal@munzur.edu.tr}

\date{ }

\begin{document}

\maketitle

\begin{abstract}
In this study, we investigate generalized quasi-Einstein structure for \NCPM s.  Firstly, we deal with elementary properties and examine, existence, and characterizations of \GQE \ \NCPM. Secondly, the generalized quasi-constant curvature of \NCPM s are studied and it is proven that a \NCPM \  with generalized quasi-constant curvature is a \GQE \ manifold. Normal metric contact pair manifolds satisfying cyclic Ricci tensor and the Codazzi type of Ricci tensor are considered and its shown that a  \GQE\ \NCPM \ does not satisfy Codazzi type of Ricci tensor. Finally, we work on \NCPM s \ satisfying certain curvature conditions related to $ \mathcal{M}- $projective, conformal, and concircular curvature tensors. We show that a \NCPM\ with generalized quasi-constant curvature is  locally isometric to the Hopf manifold $  S^{2n+1}(1) \times S^1 $.
\end{abstract}
\begin{keywords*}
	\GQE,\ Contact pairs, generalized quasi-constant curvature, \\$ \mathcal{M}- $projevtice curvature tensor 
\end{keywords*}
\section{Introduction}
\quad An Einstein manifold, which comes from the solution of Einstein fields equation, is defined as Ricci tensor $ Ric $ of type $ (0, 2) $ is non-zero and proportional to the metric tensor. Since Einstein manifolds have important differential geometric properties as well as their physical applications, they are studied by geometers in a wide perspective. A generalization of Einstein manifolds is quasi-Einstein manifolds. A Riemannian manifold $ M $ is called quasi-Einstein if the Ricci curvature tensor has the following form: 
\begin{equation*}
Ric(X_1,X_2)=\lambda g(X,Y)+\beta\omega(X_1)\omega(X_2)
\end{equation*}
for all $ X_1,X_2 \in \Gamma(TM) $, where $ \lambda, \beta $ are scalars and $ \omega $ is a $ 1- $form. In contact geometry $ \eta- $Einstein manifolds could be seen as a special case of quasi-Einstein manifolds. Quasi-Einstein manifolds arose during the study of exact solutions of the
Einstein field equations as well as during considerations of quasi-umbilical hypersurfaces. For instance, the Robertson-Walker space-times are quasi-Einstein manifolds. For more details about quasi-Einstein manifolds see \cite{de2004quasi,ghosh2006certain,tripathi2007n}. 
  \par
  The generalization of quasi-Einstein manifolds has been presented in different perspectives. One of them was given by Chaki in \cite{chaki2001generalized} and another was presented by Catino \cite{catino2012generalized} with to generalize the concepts of the Ricci soliton, the Ricci almost soliton and quasi-Einstein manifolds. The third definition of \GQE\ manifolds were given by De and Ghosh \cite{de2004generalized} as follow,
  \begin{equation*}
  Ric(X_1,X_2)=\lambda g(X,Y)+\beta\omega(X_1)\omega(X_2)+\mu\eta(X_1)\eta(X_2)
  \end{equation*}
  where $ \omega, \mu $ are two non-zero $ 1- $ forms and $   \lambda, \beta$
 are certain non-zero scalars \cite{chaki2000quasi}. The unit and orthogonal vector fields $ \xi_1 $ and $ \xi_2 $ corresponding to the $ 1- $forms $ \omega  $ and $\eta $   are defined by
  $ g(X_1,\xi_1)=\omega(X_1), \ g(X_2,\xi_2)=\eta(X_1) $, respectively  \cite{de2004generalized}. The geometric properties of \GQE \ manifolds have been studied in \cite{de2004generalized,sular2012characterizations,de2011existence,prakasha2014some,de2015generalized}.
  A \GQE\ manifold , in addition to its geometrical features, it also has physical applications in general relativity \cite{mallick2016class,guler2016study,de2019generalized}. Complex $ \eta-$Einstein manifolds are a special case of \GQE\ manifolds (see \cite{turgutvanlicomplexeta2017}). \par 
  
  In \cite{bande2005contact} Bande and Hadjar defined a new contact structure on a differentiable $ (m=2p+2q+2)-$dimensional manifold $ M $ with two $ 1- $forms $ \alpha_1,\alpha_2 $.  This structure was firstly studied by Blair, Ludden and Yano \cite{blair1974geometry} as the name of bicontact manifolds. Bande and Hadjar considered a special type of $ f- $structure with complementary frames related to these contact forms and they obtained associated metric. A differentiable manifold with this structure is called metric contact pair (MCP) manifold . Riemannian geometry of MCP manifolds is given in \cite{bande2010normal,bande2009contact}. \par
This paper is on applications of \GQE\ manifolds in contact geometry. We consider the \NCPM\ admits \GQE\ structure and call this type of manifold as \GQE\ \NCPM s. After presenting definitions and basic properties, we examine the existence of such manifolds. Also, we present a characterization of \GQE \ \NCPM s.\ Moreover, we consider the notion of generalized quasi-constant curvature for \NCPM s and we obtain some results on the sectional curvature.  We investigate a generalized quasi-Einstein normal metric contact pair manifold under some conditions for Ricci tensor.
Finally, we study on $ M-$projective curvature tensor, conformal curvature tensor and concircular curvature tensor on a \NCPM\ and we gave some results under flatness and symmetry conditions.

\section{Preliminaries}
\quad Contact pairs were defined by Bande and Hadjar \cite{bande2005contact} in 2005, for details see \cite{bande2005contact,bande2009contact,bande2010normal}. In this section we give some fundamental facts about contact metric pair manifolds.  

\begin{definition}
	A pair of two $  1-$forms, $ (\alpha_1,\alpha_2) $ on an $ (m=2p+2q+2) $ -dimensional differentiable manifold $ M $, is called contact pair of type $ (p,q) $ if we have
	\begin{itemize}
		\item $ \alpha_1 \wedge (d\alpha_1)^p\wedge\alpha_2 \wedge (d\alpha_2)^q\neq0 $,
		\item  $ (d\alpha_1)^{p+1}=0 $ and $ (d\alpha_2)^{q+1}=0, $
	\end{itemize}
	where $ p,q $ are positive integers \cite{bande2005contact}. 
\end{definition}
For $ 1- $forms $ \alpha_1 $ and $ \alpha_2 $, we have two integrable subbundle of $ TM $; $ \mathcal{D}_{1}=\{X:\alpha_1(X)=0, X\in \Gamma(TM)\} $ and  $ \mathcal{D}_{2}=\{X:\alpha_2(X)=0, X\in \Gamma(TM)\} $. Then we have two characteristic foliations of $ M $, denoted $ \mathcal{F}_1=\mathcal{D}_1 \cap kerd\alpha_1$ and $ \mathcal{F}_2= \mathcal{D}_2 \cap kerd\alpha_2$ respectively. $ \mathcal{F}_1 $ and $ \mathcal{F}_2 $ are $( 2p+1) $ and $ (2q+1)-$dimensional contact manifolds with contact form induced by $ \alpha_2 $ and $ \alpha_1 $. Thus we can define $ (2p+2q)- $dimensional horizontal subbundle $ \mathcal{H}=ker\alpha_1\cap ker\alpha_2$.
To a contact pair $ (\alpha_1,\alpha_2) $ of type $ (p,q) $ there are associated two commuting vector fields $ Z_1 $ and $ Z_2 $, called Reeb vector fields of the pair, which are uniquely determined by the following equations:
\begin{eqnarray*}
	&\alpha_1(Z_1)=\alpha_2(Z_2)=1, \ \alpha_1(Z_2)=\alpha_2(Z_1)=0\\&
	i_{Z_1}d\alpha_1 =i_{Z_1}d\alpha_2=i_{Z_2}d\alpha_2=0
\end{eqnarray*}
where $ i_X $ is the contraction with the vector field X. In particular, since the Reeb vector fields
commute, they determine a locally free $ \mathbb{R}^2 $-action, called the Reeb action. \par 
The tangent bundle of $ (M,(\alpha_1,\alpha_2)) $ can be split into as follow. For the two subbundle of $ TM $
\begin{eqnarray*}
	T\mathcal{G}_i=ker d\alpha_i\cap ker \alpha_1 \cap ker \alpha_2, \ i=1,2
\end{eqnarray*}
and we can write
\begin{equation*}
T\mathcal{F}_i=T\mathcal{G}_i\oplus \mathbb{R}Z_1.
\end{equation*}
Therefore we get $ TM=T\mathcal{G}_1\oplus T\mathcal{G}_2\oplus \mathbb{R}Z_1\oplus \mathbb{R}Z_2$. The horizontal subbundle can be written as  $ \mathcal{H}=T\mathcal{G}_1\oplus T\mathcal{G}_2 $. Also, we write  $ \mathcal{V}= \mathbb{R}Z_1\oplus \mathbb{R}Z_2$, we call $ \mathcal{V} $ is  vertical subbundle of $ TM. $ \par
Let $ X $ be an arbitrary vector field on $ M $. We can write $ X=X^{\mathcal{H}}+X^{\mathcal{V}} $, where $ X^{\mathcal{H}}, X^{\mathcal{V}} $ horizontal and vertical component of $ X $ respectively. We call a vector field $ X $ as horizontal vector field if $ X \in \Gamma(\mathcal{H})$, and vertical vector field if $ X \in \Gamma(\mathcal{V})$. \par 
An almost contact structure is related to contact form and it is defined with tensorial viewpoint. Many geometric properties of contact manifolds examined by this contact structure. Similarly an almost contact pair structure has been defined as follow by Bande and Hadjar\cite{bande2009contact}. 
\begin{definition}
	An almost contact pair structure on a $ (m=2p+2q+2)- $dimensional differentiable manifold $ M $ is a triple $ \alpha_1, \alpha_2, \phi $, where $ (\alpha_1, \alpha_2) $ is a contact pair and $ \phi $ is a $ (1,1) $ tensor field such that:
	\begin{eqnarray}
	\phi^2=-I+\alpha_1\otimes Z_1+\alpha_2\otimes Z_2
	,\ \ \phi Z_1=\phi Z_2=0. 
	\end{eqnarray}
	The rank of $ \phi $ is $ (2p+2q) $ and $ \alpha_i(\phi)=0 $ for $ i=1,2 $ \cite{bande2009contact}. 
\end{definition}
The endomorphism $\phi $ is said to be decomposable, i.e $ \phi=\phi_1+\phi_2 $, if $ T\mathcal{F}_i $ is invariant under $ \phi $. If $ \phi $ is decomposable then  $ (\alpha_i, Z_i, \phi) $ induce an almost contact structure on $ \mathcal{F}_j$ for $ i\neq j$ \cite{bande2005contact}. The decomposability of $ \phi $ do not satisfy for every almost contact pair structure. An example was given in \cite{bande2005contact}, has an almost contact pair structure but $ \phi $ is not decomposable. In this study we assume that $ \phi $ is decomposable. 
\begin{definition}
	Let $ (\alpha_1,\alpha_2,Z_1,Z_2,\phi) $ be an almost contact pair structure on a manifold $ M $. A Riemannian metric $ g $ is called 
	\begin{enumerate}
		\item [$ \bullet $] compatible if $ g(\phi X_1,\phi X_2)=g(X_1,X_2)-\alpha_1(X_1)\alpha_1(X_2)-\alpha_2(X_1)\alpha_2(X_2) $ for all $ X_1,X_2 \in TM $,
		\item [$ \bullet $] associated if $ g(X_1,\phi X_2)=(d\alpha_1+d\alpha_2)(X_1,X_2) $ and $ g(X_1,Z_i)=\alpha_i(X_1) ,$ for $ i=1,2 $ and for all $ X_1,X_2 \in \Gamma(TM) $
	\end{enumerate}
\end{definition}
\noindent
$ 4- $tuple $ (\alpha_1,\alpha_2,\phi, g) $ is called a metric almost contact pair on a manifold $ M $ and g is an associated metric with respect to contact pair structure $ (\alpha_1,\alpha_2,\phi) $. We recall $ (M,\phi,Z_1,Z_2,\alpha_1,\alpha_2,g) $ is a metric almost contact pair manifold.\par
We have following properties for a metric almost contact pair manifold $ M $ \cite{bande2005contact}: 
\begin{eqnarray}
&g(Z_i,X)=\alpha_i(X), \ \  g(Z_i,Z_j)=\delta_{ij} \\& \nabla_{Z_i}Z_j=0, \ \nabla_{Z_i}\phi=0
\end{eqnarray}
and for every $ X $ tangent to $ \mathcal{F}_i \,\ i=1,2$ we have
\begin{equation*}
\nabla_ X Z_1=-\phi_1  X , \ \ \nabla_ X Z_2=-\phi_2  X 
\end{equation*}
where $ \phi=\phi_1+\phi_2 $. Normality of metric almost contact pair manifold is given by Bande and Hadjar \cite{bande2010normal}. For a metric almost contact pair manifold, we have two almost complex structure: 
\begin{equation*}
J=\phi-\alpha_2\otimes Z_1+\alpha_1\otimes Z_2,\ \ T=\phi+\alpha_2\otimes Z_1-\alpha_1\otimes Z_2.
\end{equation*}
\begin{definition}
	A metric almost contact pair manifold is said to be normal if $ J $ and $ T $ are integrable \cite{bande2010normal}. 
\end{definition}
\begin{theorem}
	Let $ (M,\phi,Z_1,Z_2,\alpha_1,\alpha_2,g) $ be a normal metric contact pair manifold then we have
	\begin{equation}\label{Phicovariantderivation}
	g((\nabla_{X_1}\phi)X_2,X_3)=\sum_{i=1}^{2}(d\alpha_i(\phi X_2,X_1)\alpha_i(X_3)-d\alpha_i(\phi X_3,X_1)\alpha_i(X_2))
	\end{equation}
	where $ X_1,X_2,X_3 $ are arbitrary vector fields on $ M $ \cite{bande2009contact}.
\end{theorem}

We  use the following statements for the Riemann curvature;
\begin{eqnarray*}
	&R(X_1,X_2)X_3=\nabla_{X_1}\nabla_{X_2}X_3-\nabla_{X_2}\nabla_{X_1}X_3-\nabla_{[X_1,X_2]}X_3,\\& R(X_1,X_2,X_3,X_4)=g(R(X_1,X_2)X_3,X_4).
\end{eqnarray*}
for all $ X_1,X_2,X_3,X_4 \in \Gamma(TM) $. Curvature properties of \NCPM s\ were given in \cite{bande2015bochner}. For  $ X_1,X_2, X_3 \in \Gamma(M) $ and $ Z=Z_1+Z_2 $ Reeb vector field, we have
\begin{eqnarray}\label{R(X,Z)Y}
R(X_1,Z)X_2&=&-g(\phi X_1,\phi X_2)Z, 
\\
\label{R(X,Y,Z,V)}
R(X_1,X_2,Z,X_3)&=&d\alpha_1(\phi X_3,X_1)\alpha_1(X_2)+d\alpha_2(\phi X_3, X_1)\alpha_2(X_2)\\&&-d\alpha_1(\phi X_3,X_2)\alpha_1(X_1)-d\alpha_2(\phi X_3,X_2)\alpha_2(X_1) \notag,\\
\label{R(X,Z)Z}
R(X_1,Z)Z&=&-\phi^2X_1.
\end{eqnarray}
Let take an orthonormal basis of $ M $ by
\begin{equation*}
S=\{e_1,e_2,...,e_p,\phi e_1,\phi e_2,...,\phi e_p,e_{p+1},e_{p+2},...,e_{p+q}, \phi e_{p+1},\phi e_{p+2},...,\phi e_{p+q} , Z_1,Z_2\}.
\end{equation*}
Then for all $ X_1\in \Gamma(TM) $ we get the Ricci curvature of $ M $ as 
\begin{equation*}
Ric(X_1,Z)=\sum_{i=1}^{2p+2q}{d\alpha_1(\phi E_i, E_i)\alpha_1(X)+d\alpha_2(\phi E_i, E_i)\alpha_2(X)}.
\end{equation*}
where $ E_i\in S , \ \ 1\leq i \leq 2p+2q+2$. So, we obtain
\begin{eqnarray}
&Ric(X_1,Z)=0,\ \ \text{for } X_1\in \Gamma(\mathcal{H}) \label{Ric(X_1,Z)},\\
&Ric(Z,Z)=2p+2q.\label{Ric(Z,Z)}\\
&Ric(Z_1,Z_1)=2p , \ Ric(Z_2,Z_2)=2q , \ Ric(Z_1,Z_2)=0.
\end{eqnarray}
\section{Existence and Characterization of Generalized Quasi-Einstein Normal Metric Contact Pairs}
\begin{definition}
	Let $ M $ be a \NCPM. \ $ M $ is called \GQE \ \NCPM \ if the Ricci curvature of $ M $ has following form;
	\begin{equation*}
	Ric(X_1,X_2)=\lambda g(X_1,X_2)+\beta \alpha_1(X_1)\alpha_1(X_2)+\mu \alpha_2(X_1)\alpha_2(X_2)
	\end{equation*}
	for functions  $ \lambda, \beta, \mu $ on $ M $ and all $ X_1,X_2 \in \Gamma(TM) $.
\end{definition}
If we set $ X_1=X_2=Z_1 $ and $ X_1=X_2=Z_2 $, respectively, we obtain $ \beta=2p-\lambda$ and $ \mu=2q-\lambda $. Thus the Ricci curvature of \GQE \ \NCPM \ is given by 
\begin{equation}\label{GQEexpression}
Ric(X_1,X_2)=\lambda g(X_1,X_2)+(2p-\lambda) \alpha_1(X_1)\alpha_1(X_2)+(2q-\lambda) \alpha_2(X_1)\alpha_2(X_2) 
\end{equation}
for all $ X_1,X_2 \in \Gamma(TM) $ and the scalar curvature is 
\begin{equation}\label{scalar}
scal=2(\lambda+1)(p+q). 
\end{equation}
Let $ X $ be an arbitrary vector field on $ M $. We can write $ X=X^{\mathcal{H} }+X^{\mathcal{V}}$. Since the Ricci curvature is a linear tensor we have 
\begin{equation*}
Ric(X_1,X_2)=Ric(X^{\mathcal {H}}_1,X^{\mathcal {H}}_2)+Ric(X^{\mathcal {V}}_1,X^{\mathcal {V}}_2).
\end{equation*}
With take consider the decomposition of tangent bundle mentioned above (see \cite{unal2019some} for details), we get 
\begin{equation*}
Ric(X^{\mathcal {V}}_1,X^{\mathcal {V}}_2)= 2p\alpha_1(X_1)\alpha_1(X_2)+2q\alpha_2(X_1)\alpha_2(X_2).
\end{equation*}	
Thus, we reach following useful result.
\begin{proposition}\label{horizontaliseinstein}
	A \NCPM\ is \GQE\ manifold if and only if the horizontal bundle is Einstein, that is for a  function $ \lambda $ on M, we have $ Ric(X^{\mathcal {H}}_1,X^{\mathcal {H}}_2)=\lambda g(X^{\mathcal {H}}_1,X^{\mathcal {H}}_2) $.
\end{proposition} 
In \cite{de2004generalized} De and Ghosh gave a theorem for the existence of a \GQE\ Riemannian manifold. 
\begin{theorem}
	If the Ricci tensor $ Ric $ of a Riemannian manifold satisfies the relation
	\begin{eqnarray}\label{existence}
	Ric(X_2,X_3)Ric(X_1,X_4)-Ric(X_1,X_3)	Ric(X_2,X_4)&=&\gamma[g(X_2,X_3)g(X_1,X_4)\\&&-g(X_1,X_3)g(X_2,X_4)]\nonumber
	\end{eqnarray}
	where $ \gamma $ is a non-zero scalar, then the manifold is a generalized quasi Einstein manifold \cite{de2004generalized}.  
\end{theorem}
Assume that (\ref{existence}) is satisfied on a \NCPM \ $ M $. By setting $ X_1=X_4=Z $ and $ X_2,X_3 \in \Gamma(\mathcal{H}) $, and from (\ref{Ric(X_1,Z)}) we have 
\begin{equation*}
Ric(X_2,X_3)=\frac{2\gamma}{m-2}g(X_2,X_3).
\end{equation*}
Thus, with consider Proposition \ref{horizontaliseinstein} $ M $ is a \GQE\ manifold. Using (\ref{scalar}), we get $ \gamma=\frac{1}{2}(scal-m+2) $ and finally we state;
\begin{corollary}
	Let $ M $ be \NCPM\ with scalar curvature $ scal\neq m-2 $. If we have the relation
	\begin{eqnarray*}
		Ric(X_2,X_3)Ric(X_1,X_4)-Ric(X_1,X_3)Ric(X_2,X_4)&=& \frac{1}{2}(scal-m+2)[g(X_2,X_3)g(X_1,X_4)\\&&-g(X_1,X_3)g(X_2,X_4)]\
	\end{eqnarray*}
	on $ M $ for all $ X_1,X_2,X_3,X_4\in \Gamma(TM) $, then $ M $ is a \GQE\ manifold. 
\end{corollary}
The sectional curvature of a Riemann manifold gives us many important geometric properties.
The characterizations of Einstein \cite{chen2000characterizations,dumitru2007einstein},  quasi-Einstein \cite{bejan2007characterization} and \GQE   \ \cite{sular2012characterizations,dumitru2012characterization}  manifolds have been obtained by using the sectional curvature of subspaces of tangent bundle. We take into consideration Theorem 2.2. in \cite{sular2012characterizations}, and thus we obtain following theorem. Firstly we should to give some fundamental facts.\par 
Let $ \pi $ be a plane section in $ T_QM $ for any $ Q\in M $. The sectional curvature of $ \pi $ is given as $ Sec(\pi) = Sec(u\wedge v) $, where $ u,v $ orthonormal vector fields . For any
$( p+q )$-dimensional subspace $ \mathcal{L} \subset T_QM, \ 2 \leq p+q \leq m$, its scalar curvature $ scal(\mathcal{L})$ is
denoted by
\begin{equation}
scal(\mathcal{L})=\sum_{1\leq i,j \leq p+q}Sec(E_i\wedge E_j)
\end{equation}
\noindent
where $ {E_1, ..., E_n} $ is any orthonormal basis of $\mathcal{L}$ \cite{chen2000some}. When $ \mathcal{L}=T_QM $, the
scalar curvature  is just the scalar curvature  $ scal(Q) $ of $ M $ at $ 
Q \in M$. 

\begin{theorem}
	An $ (m=2p+2q+2)- $dimensional \NCPM\ is a \GQE \ manifold if and only if there exist a function $ \lambda $ on $ M $ satisfying 
	\begin{eqnarray*}
		scal(P)+p+q-\lambda&=&\sec(P^{\bot}), \ \ Z_1,Z_2 \in T_QP^{\bot}\\
		scal(N)+p+q&=&\sec(N^{\bot}), \ \ Z_1,Z_2 \in T_QN^{\bot}\\
		scal(R)+q-p&=&\sec(P^{\bot}), \ \ Z_1\in T_QR,Z_2 \in T_QR^{\bot}
	\end{eqnarray*}
	where $ (p+q+1)-$plane sections $ P,R $ and $ (p+q)-$plane section $ N $; $ P^{\bot},N^{\bot} $ and $ R^{\bot} $ denote the orthogonal complements of $ P,N $ and $ R $ in $ T_QM $, respectively. 
\end{theorem}

\section{Generalized Quasi-constant Curvature of Normal Metric Contact Pair Manifolds}
\quad The notion of quasi-constant curvature was defined by Chen and Yano \cite{chen1972hypersurfaces}. De and Ghosh generalized this notion for a Riemannian manifold. Let $ M $ be a \NCPM. \ $ M $ is called a \NCPM \ of generalized quasi-constant curvature if the Riemannian curvature tensor of $ M $ satisfying; 
\begin{eqnarray}
R(X_1,X_2,X_3,X_4)&=&A[g(X_2,X_3)g(X_1,X_4)-g(X_1,X_3)g(X_2,X_4)]\label{gen.quasiconscurv}\\&& B[g(X_1,X_4)\alpha_1(X_2)\alpha_1(X_3)-g(X_1,X_3)\alpha_1(X_2)\alpha_1(X_4)\notag\\&&+g(X_2,X_3)\alpha_1(X_1)\alpha_1(X_4)-g(X_2,X_4)\alpha_1(X_1)\alpha_1(X_3)]\notag\\&& C[g(X_1,X_4)\alpha_2(X_2)\alpha_2(X_3)-g(X_1,X_3)\alpha_2(X_2)\alpha_2(X_4)\notag\\&&+g(X_2,X_3)\alpha_2(X_1)\alpha_2(X_4)-g(X_2,X_4)\alpha_2(X_1)\alpha_2(X_3)]\notag
\end{eqnarray} 
for all $ X_1,X_2,X_3,X_4 \in \Gamma(TM)$, where $ A,B $ and $ C $ are scalar functions. \par 
\begin{proposition}
	Let $ M $ be a \NCPM \ of generalized quasi-constant curvature. Then, we have followings;
	\begin{enumerate}
		\item [$ \bullet $] the sectional curvature of horizontal bundle is $ A $,
		\item [$ \bullet $] the sectional curvature of plane section spanned by $ X\in \Gamma(\mathcal{H}) $ and $ Z $ is $ 2A+B+C $,
		\item [$ \bullet $] the sectional curvature of plane section spanned by $ X\in \Gamma(\mathcal{H}) $ and $ Z_1,Z_2$ is $ A+B$ and $ A+C $, respectively. 
	\end{enumerate}
	
\end{proposition}  
\begin{proof}
	Let take $ X_1=X_4=X , X_2=X_4=X^{\prime}$, where $ X,X^{\prime} $ unit and mutually orthogonal horizontal vector fields. Then from (\ref{gen.quasiconscurv}) we get 
	\begin{eqnarray*}
		sec(X,X^{\prime})&=&A[g(X^{\prime},X^{\prime})g(X,X)-g(X,X^{\prime})g(X^{\prime},X)]\\&=&A
	\end{eqnarray*}
	For $ X_1=X_4=X , X_2=X_4=Z$ for unit horizontal vector field $ X $ we get 
	\begin{eqnarray*}
		sec(X,Z)&=&A[g(Z,Z)g(X,X)-g(X,X^{\prime})g(X^{\prime},X)-g(X,Z)g(Z,X)]\\&&+B(g(X,X)\alpha_1(Z)\alpha_1(Z)+Cg(X,X)\alpha_2(Z)\alpha_2(Z)\\&=&2A+B+C
	\end{eqnarray*}
	By following same steps one can obtain other results. 
\end{proof}
From above proposition we get 
\begin{corollary} 
	\NCPM \ of generalized quasi-constant curvature, we have 
	\begin{equation*}
	sec(X,Z)=sec(X,Z_1)+sec(X,Z_2)
	\end{equation*}
	for any horizontal and unit vector field $ X $. 
\end{corollary}
We can also characterize a \GQE \ \NCPM \  of generalized quasi-constant curvature. We get following result. 
\begin{theorem}
	A \NCPM \ of generalized quasi-constant curvature is a \GQE \ manifold with coefficients $\lambda_1=A(m-1)+B+C, \ \lambda_2=B(m-2)$, and $ \lambda_3=C(m-2)$. 
\end{theorem}
\begin{proof}
	Let $ M $ be \NCPM \ of generalized quasi-constant curvature.  Take an orthonormal basis of $ M $ as	
	\begin{eqnarray*}
		S=\{e_1,e_2,...,e_p,\phi_1e_1,\phi_1e_2,...,\phi_1e_p, e_{p+1} ,e_{p+2},...,e_{p+q}, \phi_2e_{p+1} ,\phi_2e_{p+2},...,\phi_2e_{p+q}, Z_1,Z_2\}.
	\end{eqnarray*}
	By taking sum of (\ref{gen.quasiconscurv}) from $ i=1 $ to $ i=2p+2q+2 $ for $ X_2=X_3=E_i \in S$, we obtain 
	\begin{eqnarray*}
		\sum_{i=1}^{2p+2q+2}R(X_1,E_i,E_i,X_4)&=&\sum_{i=1}^{2p+2q+2}\{A[g(E_i,E_i)g(X_1,X_4)-g(X_1,E_i)g(E_i,X_4)]\\&& B[g(X_1,X_4)\alpha_1(E_i)\alpha_1(E_i)-g(X_1,E_i)\alpha_1(E_i)\alpha_1(X_4)\notag\\&&+g(E_i,E_i)\alpha_1(X_1)\alpha_1(X_4)-g(E_i,X_4)\alpha_1(X_1)\alpha_1(E_i)]\notag\\&& C[g(X_1,X_4)\alpha_2(E_i)\alpha_2(E_i)-g(X_1,E_i)\alpha_2(E_i)\alpha_2(X_4)\notag\\&&+g(E_i,E_i)\alpha_2(X_1)\alpha_2(X_4)-g(E_i,X_4)\alpha_2(X_1)\alpha_2(E_i)]\}\notag. 
	\end{eqnarray*} 
	For $ 1\leq i \leq 2p+2q $ since  $ \alpha_j(E_i)=0 $, $j =1,2 $ and $ \sum_{i=1}^{2p+q+2}g(X_1,E_i)g(E_i,X_2)=g(X_1,X_2)$ we get 
	\begin{eqnarray*}\label{genconstantresult}
		Ric(X_1,X_4)&=&[A(m-1)+B+C]g(X_1,X_4)+B(m-2)\alpha_1(X_1)\alpha_1(X_4)\\&&+C(m-2)\alpha_2(X_1)\alpha_2(X_4)\notag
	\end{eqnarray*}
	which completes the proof. 
\end{proof}
\section{Normal Metric Contact Pair Manifold Satisfying Cyclic Ricci Tensor}
\quad A Riemann manifold satisfies cyclic Ricci tensor if we have 
\begin{equation*}
(\nabla_{X_1}Ric)(X_2,X_3)+(\nabla_{X_2}Ric)(X_3,X_1)+(\nabla_{X_3}Ric)(X_1,X_2)=0
\end{equation*}
for all $ X_1,X_2,X_3 \in \Gamma(TM) $. De and Mallick proved that a \GQE\ Riemann manifold satisfies cyclic Ricci tensor if generators of the manifolds are Killing \cite{de2015generalized}. \par
The characteristic vector fields of a \NCPM\ $ Z_1,Z_2 $ are Killing vector fields \cite{bande2013symmetry}. Thus, by easy computations, we get 
\begin{eqnarray}
(\nabla_{X_1}\alpha_1)X_2+(\nabla_{X_2}\alpha_1)X_1=0 \label{covdervalpha1},\ \  (\nabla_{X_1}\alpha_2)X_2+(\nabla_{X_2}\alpha_2)X_1=0.
\end{eqnarray}
Let $ M $ be a \GQE\ \NCPM. \ It is well known that 
\begin{equation*}
(\nabla_{X_1}Ric)(X_2,X_3)=\nabla_{X_1}Ric(X_2,X_3)-Ric(\nabla_{X_1}X_2,X_3)-Ric(X_2,\nabla_{X_1}X_3). 
\end{equation*}
Then, from (\ref{GQEexpression}) we obtain 
\begin{eqnarray}\label{nablaRic}
(\nabla_{X_1}Ric)(X_2,X_3)&=&X_1[\lambda]g(\phi X_2,\phi X_3)\\&&+(2p-\lambda)[\left((\nabla_{X_1}\alpha_1)X_3\right)\alpha_1(X_3)+\left((\nabla_{X_1}\alpha_1)X_2\right)\alpha_1(X_2)] \notag \\&&(2q-\lambda)[\left((\nabla_{X_1}\alpha_1)X_3\right)\alpha_2(X_2)+\left((\nabla_{X_1}\alpha_1)X_2\right)\alpha_2(X_3)]\notag 
\end{eqnarray}
where $ X_1[\lambda] $ is the derivation of $ \lambda $ in the direction of $ X_1 $. Thus from (\ref{covdervalpha1}) we obtain 
\begin{eqnarray*}
	(\nabla_{X_1}Ric)(X_2,X_3)+(\nabla_{X_2}Ric)(X_3,X_1)+(\nabla_{X_3}Ric)(X_1,X_2)&&\\=
	X_1[\lambda]g(\phi X_2,\phi X_3)+X_2[\lambda]g(\phi  X_3,\phi X_1)+X_3[\lambda]g(\phi X_1, \phi X_2).
\end{eqnarray*}
As a consequence , we can state following theorem.
\begin{theorem}\label{cyclicRicciTen.Theo}
	Let $ M $ be \GQE \ \NCPM. \ If $ \lambda  $ is constant then $ M $ satisfies cyclic Ricci tensor.  
\end{theorem}
A Riemann manifold satisfies Codazzi type of Ricci tensor if 
\begin{equation*}
(\nabla_{X_1}Ric)(X_2,X_3)-(\nabla_{X_2}Ric)(X_1,X_3)=0
\end{equation*}
for all $ X_1\,X_2 $ vector fields on $ M $. In \cite{de2015generalized}
it has been proven that if a \GQE\ Riemann manifold satisfies Codazzi type of Ricci tensor, then the associated 1-forms are closed. \par
Suppose that Ricci tensor $ Ric $ of a \NCPM\ $ M $ is Codazzi type. Then, from (\ref{covdervalpha1}) and (\ref{nablaRic}) we obtain  
\begin{eqnarray*}
	&&(2p-\lambda)[((\nabla_{X_1}\alpha_1)X_2-(\nabla_{X_2}\alpha_1)X_1)\alpha_1(X_3) \\&&+ ((\nabla_{X_1}\alpha_1)X_3\alpha_1(X_2)+(\nabla_{X_2}\alpha_1)X_3\alpha_1(X_1))]\\&&+(2q-\lambda)[((\nabla_{X_1}\alpha_1)X_2-(\nabla_{X_2}\alpha_1)X_1)\alpha_2(X_3) \\&&+ ((\nabla_{X_1}\alpha_1)X_3\alpha_2(X_2)+(\nabla_{X_2}\alpha_1)X_3\alpha_2(X_1))]=0.
\end{eqnarray*}
Let take $ X_3=Z_1 $, then we get 
\begin{equation*}
(2p-\lambda)((\nabla_{X_1}\alpha_1)X_2-(\nabla_{X_2}\alpha_1)X_1)=0
\end{equation*}
which implies $ \lambda=2p $ or $ (\nabla_{X_1}\alpha_1)X_2-(\nabla_{X_2}\alpha_1)X_1=0 $. If $ \lambda=2p $ the manifold is not \GQE ,\ so this case is not possible. In the other case we obtain 
\begin{equation*}
0=(\nabla_{X_1}\alpha_1)X_2-(\nabla_{X_2}\alpha_1)X_1=d\alpha_1(X_1,X_2)=0
\end{equation*}
and so $ \alpha_1 $ is closed. Similarly by choosing $ X_3=Z_2  $ we obtain $ \alpha_2$ is closed. As we know contact pairs $ (\alpha_1,\alpha_2) $ are not closed. So our assumption is not valid. Finally we conclude that 
\begin{theorem}
	A \GQE\ \NCPM \ do not satisfy Codazzi type of Ricci tensor. 
\end{theorem}
\section{Normal Metric Contact Pair Manifold Satisfying Certain curvature Conditions }
\quad  Curvature tensors give us many geometric properties of contact manifolds. Some properties of \NCPM\ satisfying certain conditions of curvature tensors were given in \cite{bande2015bochner,unal2019some}. In this section we examine the $ \mathcal{M}- $projective curvature tensor $ \mathcal{W} $, conformal curvature tensor $ \mathcal{C} $ and concircular curvature tensor $ \mathcal{Z} $ on a \NCPM. \par
Let $ M $ be an $ (m=2p+2q+2)- $dimensional \NCPM. \ $ \mathcal{M}- $projective curvature tensor of $ M $ is given by ,
\begin{eqnarray}\label{m-projective}
\mathcal{W}( X_1,X_2)X_3&=&R( X_1,X_2)X_3-\frac{1}{2(m-1)}[Ric(X_2,X_3)X_1\\&&-Ric(X_1,X_3)X_2+g(X_2,X_3)QX_1  -g(X_1,X_3)QX_2]\notag
\end{eqnarray}
for $ X_1,X_2,X_3 \in \Gamma(TM)$, where $ Q $ is Ricci operator is given by $ Ric(X_1,X_2)=g(QX_1,X_2) $  \cite{pokhariyal1970curvature}.  $ \mathcal{M}- $projective curvature tensor on manifolds with different structures studied by many authors \cite{ayar2019m,prakasha2020M,zengin2014m}.\par 
From (\ref{m-projective}), we have
\begin{eqnarray}
\mathcal{W}( X_1,Z)Z&=&\frac{m}{2m-1}X_1-\frac{1}{m-1}QX_1 \label{M(X,Z)Z}\\
\mathcal{W}( X_1,X_2)Z&=&R(X_1,X_2)Z \label{M(X,Y)Z}\\
\mathcal{W}( X_1,Z)X_2&=&\left[\frac{(2m-1)(m-2)}{2(m-1)}g(X_1,X_2)+\frac{1}{2(m-1)}Ric(X_1,X_2)\right]Z\label{M(X,Z)Y}
\end{eqnarray}
for $ X_1,X_2,X_3 \in \Gamma(\mathcal{H}) $. Also, since $ Ric(X_1,X_2)=g(QX_1,X_2) $, where $ Q $ is the Ricci operator, we have  
\begin{eqnarray}
\mathcal{W}( X_1,X_2, X_3 , X_4)&&=R( X_1,X_2, X_3 , X_4)-\frac{1}{2(m-1)}[Ric(X_2,X_3)g(X_1,X_4)\label{m-projec}\\&&-Ric(X_1,X_3)g(X_2,X_4)+g(X_2,X_3)Ric(X_1,X_4)\notag \\&& -g(X_1,X_3)Ric(X_2,X_4)]\notag.
\end{eqnarray}
for all $X_1,X_2,X_3 \in \Gamma(M)  $. 
$ M $  is called $ \mathcal{M}- $projective flat if $ \mathcal{W} $ vanishes identically on $ M $. 
\begin{theorem}
	A \GQE \ \NCPM\ is $ \mathcal{M}- $projective flat if and only if it is of generalized quasi-constant curvature. 
\end{theorem}
\begin{proof}
	Suppose that $ M $ is a \GQE \ manifold. Then, from (\ref{GQEexpression}) and (\ref{m-projective}) we have 
	\begin{eqnarray}\label{mprojectflat}
	\mathcal{W}( X_1,X_2, X_3 , X_4)&=&R(X_1,X_2,X_3,X_4)\\&&-\frac{\lambda}{m-1}[g(X_2,X_3)g(X_1,X_4)-g(X_1,X_3)g(X_2,X_4)]\notag\\&&- \frac{2p-\lambda}{2(m-1)}[g(X_1,X_4)\alpha_1(X_2)\alpha_1(X_3)-g(X_1,X_3)\alpha_1(X_2)\alpha_1(X_4)\notag\\&&+g(X_2,X_3)\alpha_1(X_1)\alpha_1(X_4)-g(X_2,X_4)\alpha_1(X_1)\alpha_1(X_3)]\notag\\&&- \frac{2q-\lambda}{2(m-1)}[g(X_1,X_4)\alpha_2(X_2)\alpha_2(X_3)-g(X_1,X_3)\alpha_2(X_2)\alpha_2(X_4)\notag\\&&+g(X_2,X_3)\alpha_2(X_1)\alpha_2(X_4)-g(X_2,X_4)\alpha_2(X_1)\alpha_2(X_3)]\notag.
	\end{eqnarray}
	Thus it is seen that $ M $ is $ \mathcal{M}- $projective flat if and only if  $ M $ is of generalized quasi-constant curvature with coefficients $ A= \frac{\lambda}{m-1}, \ B=\frac{2p-\lambda}{2(m-1)}$  and $ C=\frac{2q-\lambda}{2(m-1)}$.
\end{proof}
The Riemann manifolds satisfying $ R(X_1,X_2).R=0 $ are called semi-symmetric, where $ R(X_1,X_2) $ acts on $ R $ as a derivation. Semi-symmetric contact manifolds were studied  by Perrone \cite{contactRiemanPerrone}.  Similarly, if $ \mathcal{W}(X_1,X_2).R=0$ then $ M $ is called $ \mathcal{M}- $projective semi-symmetric.  $ \mathcal{W}(X_1,X_2).R$ is defined as 
\begin{eqnarray}\label{M.Rgeneral}
(\mathcal{W}(X_1,X_2).R)(X_3,X_4)X_5&= \mathcal{W}(X_1,X_2).R(X_3,X_4)X_5-R(\mathcal{W}(X_1,X_2),X_3,X_4)X_5\\&
-R(X_3,\mathcal{W}(X_1,X_2)X_4)X_5-R(X_3,X_4)\mathcal{W}(X_1,X_2)X_5\notag 
\end{eqnarray}
for all $ X_1,X_2,X_3,X_4,X_5\in \Gamma(TM) $.  Also, we have 
\begin{eqnarray}\label{M.Ricgeneral}
(\mathcal{W}(X_1,X_2).Ric)(X_3,X_4)=-Ric(\mathcal{W}(X_1,X_2),X_3)-Ric(X_3,\mathcal{W}(X_1,X_2)X_4). 
\end{eqnarray}
If $ \mathcal{W}(X_1,X_2).Ric=0 $ then $ M $ is called $ \mathcal{M}- $projective Ricci semi-symmetric. 
\begin{theorem}
	A \NCPM \ is $ \mathcal{M}- $projective semi-symmetric if and only if $ M $ is \GQE \ manifold. 
\end{theorem}
\begin{proof}
	From (\ref{M.Rgeneral}) and using (\ref{M(X,Z)Z}), (\ref{M(X,Y)Z}) and (\ref{M(X,Z)Y}) we obtain 
	\begin{eqnarray*}
		(\mathcal{W}(X_1,Z).R)(X_3,X_4)X_5&=&KR(X_1,X_3,X_4,X_5)Z+LR(X_3,X_4,X_5,QX_1)Z\\&&-(Kg(X_1,X_3)+LRic(X_1,X_3))g(X_4,X_5)Z\\&&+(Kg(X_1,X_5)+LRic(X_1,X_5))R(X_3,X_4)Z
	\end{eqnarray*}
	where $ K=\frac{(2m-1)(m-2)}{2(m-1)} $ and $ L= \frac{1}{2(m-1)}$. \par
	Let take $ X_1,X_3,X_4,X_5 $ horizontal vector fields and $ X_4=Z $, then from (\ref{R(X,Z)Y}),  (\ref{R(X,Y,Z,V)}), (\ref{R(X,Z)Z}), (\ref{Ric(X_1,Z)}) and (\ref{Ric(Z,Z)}), then, we get 
	\begin{eqnarray*}
		(\mathcal{W}(X_1,Z).R)(X_3,Z)X_5=-(Kg(X_1,X_5)+LRic(X_1,X_5))X_3.
	\end{eqnarray*}
	Thus, we conclude that $ (\mathcal{W}(X_1,Z).R)(X_3,X_4)X_5=0 $ if and only if horizontal bundle of  $ M $ is Einstein. Take into consideration of  Proposition \ref{horizontaliseinstein}, we obtain 
	\begin{eqnarray}\label{GQEinproof}
	Ric(X_1,X_5)=-\frac{K}{L}g(X_1,X_5)
	+(2p+\frac{K}{L})\alpha_1(X_1)\alpha_1(X_5)+(2q+\frac{K}{L})\alpha_2(X_1)\alpha_2(X_5) 
	\end{eqnarray}
	So, the manifold is \GQE . 
\end{proof}
\begin{theorem}
	An $ (m=2p+2q+2) $-dimensional \NCPM \ admitting an $ \mathcal{M}- $projective curvature
	tensor and a nonzero Ricci tensor $ Ric $ satisfies the equality $ \mathcal{W}.Ric =0 $ if and only $ M $ is \GQE \ manifold.
\end{theorem}
\begin{proof}
	For $ X_1,X_3,X_4 \in \Gamma(TM) $ from (\ref{M.Ricgeneral}) we get 
	\begin{eqnarray*}
		(\mathcal{W}(X_1,Z).Ric)(X_3,X_4)&=&-Ric([Kg(X_1,X_3)+LRic(X_1,X_3)]Z,X_4)\\&&-Ric(X_3,[Kg(X_1,X_4)+LRic(X_1,X_4)]Z). 
	\end{eqnarray*}
	Let take $ X_1,X_4 $ horizaontal vector fields and $ X_3=Z $ from (\ref{Ric(X_1,Z)}), (\ref{Ric(Z,Z)}), we obtain 
	\begin{eqnarray*}
		(\mathcal{W}(X_1,Z).Ric)(X_3,X_4)=-(2p+2q)[Kg(X_1,X_4)+LRic(X_1,X_4)].
	\end{eqnarray*}
	Therefore,  $ (\mathcal{W}(X_1,Z).Ric)(X_3,X_4)=0 $ if and only if horizontal bundle is Einstein. From Proposition (\ref{horizontaliseinstein}) we get (\ref{GQEinproof}), which completes the proof. 	
\end{proof}
Conformal $ \mathcal{C} $ and concircular $ \mathcal{Z} $ on a $(m=2p+2q+2)$-dimensional normal contact metric pair manifold are given by: 
\begin{eqnarray*}
	\mathcal{C}( X_1,X_2) X_3 &=&R( X_1,X_2) X_3\\
	&&+\frac{scal}{(
		m-1) (m-2)}\left( g(X_2,X_3)X_1-g(X_1,X_3)X_2\right)  \label{conformal} \notag\\
	&&+\frac{1}{m-2}( g( X_1,X_3) QX_2-g( X_2,X_3) QX_1\notag\\
	&&+Ric(
	X_1,X_3) X_2-Ric( X_2,X_3) X_1) , \notag\\
	\mathcal{Z}( X_1,X_2) X_3&=&R( X_1,X_2) X_3-\frac{scal}{	m (m-1) }[ g(X_2,X_3)X_1-g(X_1,X_3)X_2] \notag 
	\label{concircular} \\
\end{eqnarray*}
where  $ X_1,X_2,X_3 \in \Gamma(TM) $. Conformal and concircular curvature tensors on contact manifolds have been studied in \cite{unal2018concircular,unal2019some,turgut2017conformal}. 

Blair, Bande and Hadjar \cite{bande2015bochner} studied on conformal flatness of \NCPM s and they proved following theorem. 
\begin{theorem}
	A conformal flat \NCPM \ is locally isometric to the Hopf manifold $  S^{2q+1}(1) \times S^1 $ \cite{bande2015bochner}.
\end{theorem}
Thus we get following results, for a \GQE \ \NCPM. 
\begin{theorem}
Let $ M $ be a \GQE \ \NCPM. If $ M $ is of generalized quasi-constant curvature  with coefficients $ A=-\frac{\lambda m-m+2}{(m-1)(m-2)} , \ B=\frac{2p-\lambda}{m-2}$ and $ C= \frac{2q-\lambda}{m-2}$, then it is  locally isometric to the Hopf manifold $  S^{2q+1}(1) \times S^1 $.
\end{theorem}
\begin{proof}
	Let $ M $ be \GQE \ \NCPM. \ Then, we have 
	\begin{eqnarray*}
		\mathcal{C}( X_1,X_2, X_3 , X_4)&=&R(X_1,X_2,X_3,X_4)\\&&+\frac{scal}{(m-1)(m-2)}[g(X_2,X_3)g(X_1,X_4)-g(X_1,X_3)g(X_2,X_4)]\\&&-\frac{2\lambda}{m-2}[g(X_2,X_3)g(X_1,X_4)-g(X_1,X_3)g(X_2,X_4)]\notag\\&& -\frac{2p-\lambda}{m-2}[g(X_1,X_4)\alpha_1(X_2)\alpha_1(X_3)-g(X_1,X_3)\alpha_1(X_2)\alpha_1(X_4)\notag\\&&+g(X_2,X_3)\alpha_1(X_1)\alpha_1(X_4)-g(X_2,X_4)\alpha_1(X_1)\alpha_1(X_3)]\notag\\&& -\frac{2q-\lambda}{m-2}[g(X_1,X_4)\alpha_2(X_2)\alpha_2(X_3)-g(X_1,X_3)\alpha_2(X_2)\alpha_2(X_4)\notag\\&&-g(X_2,X_4)\alpha_2(X_1)\alpha_2(X_3)]\notag.
	\end{eqnarray*}	
	Suppose that, $ M $ is of generalized quasi-constant curvature  with coefficients $ A=-\frac{\lambda m-m+2}{(m-1)(m-2)} , \ B=\frac{2p-\lambda}{m-2}$ and $ C= \frac{2q-\lambda}{m-2}$  . Then, we get $ \mathcal{C}=0 $  which means that $ M $ is conformal flat. Thus, from Theorem 6.4 $ M $ is locally isometric to the Hopf manifold $  S^{2q+1}(1) \times S^1 $. 
\end{proof}
By using the definition of $ \mathcal{M}- $projective curvature tensor and conformal curvature tensor, we have 
\begin{eqnarray}\label{relationWMC}
\mathcal{C}( X_1,X_2) X_3&=&\frac{2(m-1)}{m-2}\mathcal{W}( X_1,X_2) X_3-\frac{m}{m-2}R( X_1,X_2,)X_3\\&&+\frac{scal}{(m-1)(m-2)}[g(X_2,X_3)X_1-g(X_1,X_3)X_2].\notag
\end{eqnarray}
If $ M $ is a $ \mathcal{M}- $projective flat \NCPM.\ , then, from (\ref{relationWMC}) , $M $ is conformal flat if and only if 
\begin{eqnarray*}
	R( X_1,X_2,)X_3=\frac{scal}{	m (m-1) }[ g(X_2,X_3)X_1-g(X_1,X_3)X_2],
\end{eqnarray*}
which means $ M $ is concircular flat. Finally, we conclude that 
\begin{theorem}
	Let $ M $ be $ \mathcal{M}- $projective flat \NCPM . \ If $ M $ is also concircular flat then it is locally isometric to Hopf manifold $  S^{2q+1}(1) \times S^1 $. 
\end{theorem}

\end{document}